\documentclass[11pt,a4paper]{amsart}
\usepackage[margin=1in]{geometry}
\usepackage[latin1]{inputenc}
\usepackage{amsbsy}
\usepackage{amscd} 
\usepackage{amsfonts}
\usepackage{amsgen} 
\usepackage{amsmath}
\usepackage{amsopn} 
\usepackage{amssymb}
\usepackage{amstext}
\usepackage{amsthm} 
\usepackage{amsxtra}
\usepackage{rotating}
\usepackage{subfig}
\usepackage{graphicx}

\usepackage[pagebackref, colorlinks = true, linkcolor = blue, urlcolor  = blue, citecolor = red]{hyperref}

\usepackage{algorithm}
\usepackage[noend]{algpseudocode}
\makeatletter
\def\BState{\State\hskip-\ALG@thistlm}
\makeatother

\theoremstyle{plain} 
\newtheorem{thm}{Theorem}[section]

\newtheorem{lem}[thm]{Lemma}

\newtheorem{prob}[thm]{Problem}
\newtheorem{conj}[thm]{Conjecture}
\theoremstyle{definition}
\newtheorem{defn}[thm]{Definition}
\newtheorem{rem}[thm]{Remark}
\newtheorem{ex}[thm]{Example}

\numberwithin{equation}{section}

\renewcommand{\theta}{\vartheta}
\renewcommand{\phi}{\varphi}
\renewcommand{\epsilon}{\varepsilon}
\renewcommand{\subset}{\subseteq}

\newcommand{\Z}{\mathbb Z}

\newcommand{\C}{\mathbb C}
\newcommand{\upsubset}{\begin{rotate}{90}$\subset$\end{rotate}}

\DeclareMathOperator{\Aut}{Aut}
\DeclareMathOperator{\QAut}{QAut}
\DeclareMathOperator{\GL}{GL}

\begin{document}
\title[Sinkhorn algorithm for quantum permutation groups]{Sinkhorn algorithm for quantum permutation groups}

\author{Ion Nechita}
\address{I.N.:~Laboratoire de Physique Th\'eorique, Universit\'e de Toulouse, CNRS, UPS, France}
\email{nechita@irsamc.ups-tlse.fr}

\author{Simon Schmidt}
\author{Moritz Weber}
\address{S.Sch. and M.W. :~Saarland University, Fachbereich Mathematik, Postfach 151150, 66041 Saarbr\"ucken, Germany}
\email{simon.schmidt@math.uni-sb.de}
\email{weber@math.uni-sb.de}

\date{\today}

\begin{abstract}
We introduce a Sinkhorn-type algorithm for producing quantum permutation matrices encoding symmetries of graphs. Our algorithm generates square matrices whose entries are orthogonal projections onto one-dimensional subspaces satisfying a set of linear relations. We use it for experiments on the representation theory of the quantum permutation group and quantum subgroups of it. We apply it to the question whether a given finite graph (without multiple edges) has quantum symmetries in the sense of Banica. In order to do so, we run our Sinkhorn algorithm and check whether or not the resulting projections commute. We discuss the produced data and some questions for future research arising from it.
\end{abstract}

\maketitle

\tableofcontents

\section{Introduction and main results}

In the 1980's, Woronowicz defined compact matrix quantum groups \cite{woronowicz1987compact} in order to provide an extended notion of symmetry in the analytic realm. The starting point is to replace compact matrix groups $G\subset \GL_n(\mathbb C)$ by their algebras $C(G)$ of continuous complex valued functions. These algebras are equipped with a Hopf algebra like structure and the crucial point is that they are commutative (with respect to pointwise multiplication). Now, a compact matrix quantum group is a \emph{noncommutative} algebra sharing certain axiomatic properties with these algebras $C(G)$. The class of compact matrix quantum groups contains the class of compact matrix groups and it provides the right analytic symmetry objects for a whole world of quantum mathematics such as quantum topology ($C^*$-algebras), quantum measure theory (von Neumann algebras), quantum probability theories (e.g. free probability theory), noncommutative geometry (in the sense of Alain Connes) and quantum information theory, just to name a few.

In the 1990's, Sh.~Wang gave a definition of quantum permutation groups $S_n^+$ \cite{wang1998quantum}. The idea is roughly to replace the entries 0 and 1 in a permutation matrix $u\in S_n\subset \GL_n(\mathbb C)$ by orthogonal projects onto subspaces of some Hilbert space; one can think of quantum permutation matrices as $n\times n$ matrices $u=(u_{ij})_{i,j=1,\ldots,n}$ whose entries $u_{ij}$ are again matrices rather than scalars 0 or 1. The crucial point is that the $u_{ij}$ are orthogonal projections summing up to the identity 1 on each row and column:
\begin{equation}\label{eq:magic-unitary}
u_{ij}=u_{ij}^2=u_{ij}^*,\quad \sum_{k=1}^n u_{ik}=\sum_{k=1}^n u_{kj}=1\quad \forall i,j=1,\ldots, n
\end{equation}
The precise definition may be found in Section \ref{Prelim}; it is given in terms of universal $C^*$-algebras $A_s(n)$. The quantum permutation group $S_n^+$ contains the classical permutation group $S_n$ as a (quantum) subgroup --- in a way, we thus have more quantum permutations than ordinary permutations.
An important task is to find representations of this $C^*$-algebra --- or, a bit more on the intuitive level, to find concrete examples of quantum permutation matrices: for instance, we would like to find matrices $u_{ij}\in M_m(\mathbb C)$ satisfying the above relations. See also \cite{banica2017flat}, \cite[Sect.~4.1]{brannan2018topological} as well as the recent links with quantum information theory \cite{lupini2017nonlocal, MR2019} for more questions in this directions.

Now, for bistochastic matrices, i.e. for orthogonal matrices $u=(u_{ij})\in O_n\subset \GL_n(\mathbb C)$ with $\sum_{k} u_{ik}=\sum_{k} u_{kj}=1$, the Sinkhorn algorithm \cite{sinkhorn1964relationship, sinkhorn1967concerning} is a useful tool to produce such matrices. The main idea is to pick an arbitrary orthogonal matrix and then to perform normalizations: In Step 1, replace all entries $u_{ij}$ by $u_{ij}':=u_{ij}/\sum_k u_{ik}$ and in Step 2, replace entries $u_{ij}'$ by $u_{ij}'':=u_{ij}'/\sum_k u_{kj}'$. Performing Steps 1 and 2 in an alternating way, one eventually obtains a bistochastic matrix. An algorithm which is adapted to quantum permutation matrices was introduced in  \cite{benoist2017bipartite,banica2017flat}: Starting with arbitrary orthogonal rank one projections $u_{ij}$, orthonormalize the rows in Step 1 and then the columns (using the new entries) in Step 2. Approximately, this eventually yields a quantum permutation matrix in the above sense. 

In this work, we introduce a new Sinkhorn-type algorithm with the help of which we may tackle the question of existence of quantum symmetries for finite graphs: Given a (directed) graph $\Gamma$ on $n$ vertices having no multiple edges, its automorphism group $\Aut(\Gamma)$ consists in all permutation matrices $\sigma\in S_n\subset \GL_n(\mathbb C)$ commuting with the adjacency matrix $A \in M_n(\{0,1\})$ of $\Gamma$. Similarly, Banica \cite{banica2005quantum} defined the quantum automorphism group $\QAut(\Gamma)$ by adding the relation
\begin{equation}\label{eq:graph-symmetry}
u A=A u
\end{equation}
to those of $S_n^+$, see \eqref{eq:magic-unitary}. Thus, $\QAut(\Gamma)$ is a quantum subgroup of $S_n^+$ containing the classical automorphism group $\Aut(\Gamma)$. We then say that $\Gamma$ has quantum symmetries, if the inclusion of $\Aut(\Gamma)$ in $\QAut(\Gamma)$ is strict \cite{banica2005quantum}. In order to prove the existence of quantum symmetries, it is sufficient to find an example of a quantum permutation matrix commuting with the adjacency matrix and having at least two entries $u_{ij}$ and $u_{kl}$ such that
\[u_{ij}u_{kl}\neq u_{kl}u_{ij}.\]
We now include the condition  \eqref{eq:graph-symmetry} in our algorithm and obtain a tool for testing a graph for quantum symmetry. The main aim of this article is to present a number of examples of cases of graphs to which we applied the algorithm and successfully confirmed the existence or non-existence of quantum symmetries. It is summarized in Table \ref{tab:graphs-q-sym}. Here, the parameter $\tau\in (0,1)$ rules the intensity of enforcing relation \eqref{eq:graph-symmetry}, while $\varepsilon$ is the tolerance regarding violations of both relations \eqref{eq:magic-unitary} and \eqref{eq:graph-symmetry}.

\begin{table}
\begin{tabular}{|c|c|c||c|c|c|c||c|c|}
\hline
\multicolumn{3}{|c||}{Graph}                                                                                         & \multicolumn{2}{c|}{Performance}                                                                                             & \multicolumn{2}{c||}{Parameters}                                  & \multicolumn{2}{c|}{Has q. symm.?} \\ \hline
\multicolumn{1}{|c|}{Name} & $|V|$ & \multicolumn{1}{c||}{\begin{tabular}[c]{@{}c@{}}$\mathrm{Aut}$ \\ \end{tabular}} & \multicolumn{1}{c|}{\begin{tabular}[c]{@{}c@{}}Avg.\\ iter.\end{tabular}} & \multicolumn{1}{c|}{\begin{tabular}[c]{@{}c@{}}Avg.\\ runtime\end{tabular}} & \multicolumn{1}{c|}{$\varepsilon$} & \multicolumn{1}{c||}{$\tau$} & \multicolumn{1}{c|}{\begin{tabular}[c]{@{}c@{}}Alg.\\ pred.\end{tabular}} & \multicolumn{1}{c|}{Proof} \\ \hline \hline
$K_4$ & 4 & $S_4$ & 106.32 & 0.23 sec. & $10^{-6}$ & 1 & Y & \cite{wang1998quantum} \\ \hline
$K_5$ & 5 & $S_5$ & 113.15 & 0.39 sec. & $10^{-6}$ & 1 & Y & \cite{wang1998quantum} \\ \hline
$K_6$ & 6 & $S_6$ & 1118.25 & 4.98 sec. & $10^{-6}$ & 1 & Y & \cite{wang1998quantum} \\ \hline
Cube $Q_3$ & 8 & $H_3$ & 92.11 & 7.25 sec & $10^{-3}$ & 0.5 & Y & \cite{banica2005quantum} \\ \hline
Petersen & 10 & $S_5$ & 402.32 & 42.4 sec. & $10^{-3}$ & 0.5 & N & \cite{schmidt2018petersen} \\ \hline
$L(Q_3)$ & 12 & $H_3$ & 71.18  & 20.53 sec. & $10^{-3}$ & 0.5 & N & --- \\ \hline
$L(C_6(2))$ & 12 & $\mathbb Z_2 \wr S_3$ & 67.89  & 36.37 sec. & $10^{-3}$ & 0.5 & N & --- \\ \hline
$\mathrm{Trun}(K_4)$ & 12 & $S_4$ & 283.87  & 132.9 sec. & $10^{-3}$ & 0.5 & N & --- \\ \hline
$K_3 \square C_4$ & 12 & $S_3\times D_4 $& 66.31 & 29.65 sec. & $10^{-3}$ & 0.5 & Y & \cite{BanBic}\footnotemark \\ \hline
$\mathrm{antip}(\mathrm{Trun}(K_4))$ & 12 & $S_4$ & 95.39 & 38.29 sec. & $10^{-3}$ & 0.5 & N & --- \\ \hline
Icosahedron & 12 & $A_5 \times \mathbb Z_2$ & 61.89  & 18 sec. & $10^{-3}$ & 0.5 & N & \cite{schmidt2019quantum} \\ \hline
co-Heawood & 14 & $\mathrm{PGL}(2,7)$ & 83.69  & 64 sec. & $10^{-1}$ & 0.5 & N & \cite{schmidt2019quantum} \\ \hline
Hamming(2,4) & 16 & $S_4 \wr \mathbb Z_2$ & 102.74  & 48.7 sec. & $10^{-1}$ & 0.5 & Y & \cite{schmidt2019quantum} \\ \hline
Shrikhande & 16 & ord.~192 &  126.73 & 54.04 sec. & $10^{-1}$ & 0.5 & N & \cite{schmidt2019quantum} \\ \hline
Cube $Q_4$ & 16 & $\mathbb Z_2 \wr S_4$ &  71.63 & 36.26 sec. & $10^{-1}$ & 0.5 & Y & \cite{banica2007hyperoctahedral} \\ \hline
Clebsch & 16 & $\Z_2^4 \rtimes S_5$ &  91.07 & 72.20 sec. & $10^{-1}$ & 0.5 & Y & \cite{schmidt2018quantuma} \\ \hline
\end{tabular}

\vspace{.3cm}

\caption{Some graphs and their quantum symmetries: numerical experiments and theory. In the cases where the result is known, the algorithm predicts the correct answer. The performance of the algorithm is measured by the average number of Sinkhorn iterations in the main function and by the average runtime, for $10^3$ independent runs.}
\label{tab:graphs-q-sym}
\end{table}

A number of questions arise in connection with Table \ref{tab:graphs-q-sym} and our Sinkhorn algorithm in general. We discuss them in Section \ref{sec:experiments}. For instance, note that our algorithm only produces quantum permutation matrices with all entries being rank one projections and its prediction is based on this special case. This implies that we may apply it only to graphs which satisfy a certain uniformly vertex transitive condition, see Definition \ref{def:vertex-transitive} and \cite{UVT} for a discussion of this property. However, even given these restrictions, our algorithm is surprisingly well-behaved and at least for the examples presented in Table \ref{tab:graphs-q-sym} it precisely matches the theoretically proved answer, if known.\footnotetext{Here $K_3 \square C_4$ denotes the Cartesian product of $K_3$ and $C_4$. This graph has quantum symmetry because of Proposition 4.1 in \cite{BanBic} and the fact that $C_4$ has quantum symmetry.} In the case of the graphs $L(Q_3), L(C_6(2)), \mathrm{Trun}(K_4)$ and $\mathrm{antip}(\mathrm{Trun}(K_4))$\footnote{By $C_6(2)$ we mean the cycle graph $C_6$, where additionally vertices at distance two in $C_6$ are connected. The graph denoted $\mathrm{Trun}(K_4)$ is the truncated tetrahedral graph, $\mathrm{antip}(\mathrm{Trun}(K_4))$ its distance-three graph. We provide the adjacency matrices of all graph appearing in Table \ref{tab:graphs-q-sym} in the supplementary material of the arXiv submission.} our algorithm predicts an answer going beyond the state of the art; we leave the full proof of the lack of quantum symmetries open.

\section{The background --- Quantum symmetries of finite graphs}
\label{Prelim}

In this section, we briefly present the mathematical background of this article: We recall the definition of compact (matrix) quantum groups, of the quantum permutation group $S_n^+$ and of quantum automorphism groups of finite graphs including the one of quantum symmetries of graphs. Readers not particularly interested in $C^*$-algebras or quantum groups may directly jump to Section \ref{SectProblem}, the formulation of the problem underlying our algorithm.

\subsection{Compact matrix quantum groups}

In 1987, Woronowicz \cite{woronowicz1987compact} defined compact matrix quantum groups as follows: A compact matrix quantum group $G=(B,u)$ is given by a unital $C^*$-algebra $B$ which is generated by the entries $u_{ij}$ of a matrix $u=(u_{ij})_{i,j=1,\ldots,n}\in M_n(B)$ such that the matrix $u$ as well as $\bar u=(u_{ij}^*)$ are invertible in $M_n(B)$ and there is a $^*$-homomorphism $\Delta:B\to B\otimes B$ mapping $u_{ij}\mapsto \sum_k u_{ik}\otimes u_{kj}$.

Let us take a look at the classical case. Let $G\subset \GL_n(\mathbb C)$ be a compact group. Let $B=C(G)$ be the algebra of continuous complex valued functions on $G$ and let $u_{ij}:G\to\mathbb C$ be the evaluation of matrices $x\in G$ at their entry $x_{ij}$. Then $(C(G), (u_{ij}))$ is a compact matrix quantum group and $\Delta$ corresponds to the matrix multiplication; here  we used the isomorphism of $C(G)\otimes C(G)$ with $C(G\times G)$. 
We infer:
\[\{G\subset \GL_n(\mathbb C) \textnormal{ compact groups}, n\in\mathbb N\}\subsetneq\{\textnormal{compact matrix quantum groups}\}.\]
See \cite{timmermann2008invitation,neshveyev2013compact,weber2017introduction} for more on this subject. A central example for our article was given by Wang in 1998 \cite{wang1998quantum}.

\begin{ex}\label{ExSnPLus}
The quantum permutation group $S_n^+$ is given by the universal unital $C^*$-algebra
\[A_s(n):=C^*\left(u_{ij}, i,j=1,\ldots,n\,\bigg|\,u_{ij}=u_{ij}^2=u_{ij}^*, \sum_{k=1}^n u_{ik}=\sum_{k=1}^n u_{kj}=1 \right).\]
\end{ex}

The quotient of $A_s(n)$ by the relations $u_{ij}u_{kl}=u_{kl}u_{ij}$ is isomorphic to $C(S_n)$, where $S_n$ is the permutation group represented by permutation matrices in $\GL_n(\mathbb C)$. We thus view $S_n$ as a quantum subgroup of $S_n^+$.
If $n\geq 4$, this inclusion is strict and we then have more quantum permutations than permutations, in a way.

\subsection{Quantum automorphism groups of graphs}

Let $\Gamma=(V,E)$ be a graph with vertex set $V=\{1,\ldots,n\}$ and edges $E\subset V\times V$, i.e. we have no multiple edges. Let $A \in M_n(\{0,1\})$ be its adjacency matrix, i.e. $A_{ij}=1$ if and only if $(i,j)\in E$. An automorphism of $\Gamma$ is an element $\sigma\in S_n$ with $(i,j)\in E$ if and only if $(\sigma(i),\sigma(j))\in E$. We may thus define the automorphism group of $\Gamma$ by
\[\Aut(\Gamma):=\{\sigma\in S_n\,|\, \sigma A = A \sigma\}\subset S_n.\]
This definition has been set in the quantum context by Banica in 2005 \cite{banica2005quantum}, see also \cite{bichon2003quantum} for a previous definition. Banica defined the quantum automorphism group $\QAut(\Gamma)$ by the quotient of $A_s(n)$ from Example \ref{ExSnPLus} by the relations $u A=A u$. We thus have:
\begin{align*}
 S_n &\qquad\subset \qquad S_n^+\\
\upsubset\; & \qquad\qquad\qquad\upsubset\;\\
 \Aut(\Gamma) &\qquad\subset \qquad \QAut(\Gamma)
\end{align*}
If the inclusion on the lower line is strict, we say that $\Gamma$ has quantum symmetries \cite{banica2005quantum}. 

\begin{defn} A graph $\Gamma$ is said to \emph{have quantum symmetries} if the quotient of $A_s(n)$ by the relation  $u A = A u$ is a  noncommutative algebra. 
\end{defn}

The question whether or not a graph has quantum symmetries is highly non-trivial. There are a number of criteria to check it and it has been settled for certain classes of graphs, but in general, one has to study this question case by case. See also \cite{schmidt2018quantum, schmidt2018petersen, schmidt2019quantum} for more on quantum automorphism groups of graphs and quantum symmetries.

\subsection{The problem}
\label{SectProblem}

In order to show that a given graph $\Gamma$ has quantum symmetries, it is enough to find a quantum permutation matrix with noncommuting entries, as we will describe in this section.

\begin{defn}\label{DefMagicUnitary}
A (concrete finite-dimensional) $n\times n$ \emph{quantum permutation matrix} (also called a \emph{magic unitary}) is a matrix $u=(u_{ij})_{i,j=1,\ldots,n}\in M_n(M_m(\mathbb C))$ for some $m\in\mathbb N$ such that all $u_{ij}\in M_m(\mathbb C)$ are orthogonal projections ($u_{ij}=u_{ij}^2=u_{ij}^*$) and we have $\sum_k u_{ik}=\sum_k u_{kj}=1$ for all $i,j=1,\ldots,n$.
\end{defn}

\begin{rem}
The (finite-dimensional) magic unitaries of Definition \ref{DefMagicUnitary} also appear in the quantum information theory literature. Here they are also called projective permutation matrices \cite[Def.~5.5]{atserias2019quantum} or matrices of partial isometries in the setting of \cite[Def.~3.2]{benoist2017bipartite}. See also \cite{lupini2017nonlocal} or the very remarkable article \cite{MR2019}.
\end{rem}

Given an adjacency matrix $A\in M_n(\{0,1\})$ and $m\in \mathbb N$, we view $A\in M_n(M_m(\mathbb C))$ as the matrix $A\otimes I_m$, i.e.~we replace all entries $A_{ij}=0$ by the $m\times m$ matrix filled with zeros, and we replace $A_{ij}=1$ by the $m\times m$ identity matrix.

\begin{lem}\label{LemQPerm}
Let $\Gamma=(V,E)$ be a finite graph with no multiple edges and let $A\in M_n(\{0,1\})$ be its adjacency matrix.
Assume that there is a quantum permutation matrix $u=(u_{ij})\in M_n(M_m(\mathbb C))$ with $u A=A u$ and two entries $u_{i_0j_0}$ and $u_{k_0l_0}$ such that
 \[u_{i_0j_0}u_{k_0l_0}\neq u_{k_0l_0}u_{i_0j_0}.\]
 Then $\Gamma$ has quantum symmetries.
\end{lem}
\begin{proof}
The quantum permutation matrix $u=(u_{ij})\in M_n(M_m(\mathbb C))$ gives rise to a $^*$-homomorphism from  the quotient $B$ of $A_s(n)$ by $u A=A u$ to $M_n(M_m(\mathbb C))$ mapping the generators of $B$ to the entries $u_{ij}\in M_m(\mathbb C)$. Thus $B$ (which is the algebra corresponding to $\QAut(\Gamma)$) is a  noncommutative algebra.
\end{proof}

We may now state the problem motivating our article.

\begin{prob}
Given a finite graph $\Gamma$ --- does it have quantum symmetries? Can we find a noncommutative quantum permutation matrix in the sense of Lemma \ref{LemQPerm}?
\end{prob}

In our article, we restrict our search to the case $m=n$ and to rank one projections for the elements $u_{ij}$. Although this seems to be quite a restriction, our algorithm still produces quite reliable results when comparing our results with the previously known answers (if they exist). More on this in Section \ref{sec:experiments}.

\section{The algorithms}

Inspired by the Sinkhorn algorithm for bistochastic matrices, the following Sinkhorn algorithm for $S_n^+$ was introduced in \cite{banica2017flat}. We shall later include the condition $u A=A u$ in order to find evidence for the (non-) existence of quantum symmetries of a given graph $\Gamma$, see Algorithm \ref{alg:Sinkhorn-QAut}.

\subsection{The algorithm for finding quantum permutation matrices: Algorithm \ref{alg:Sinkhorn-Sn-plus}}

The main idea behind the Sinkhorn algorithm for finding quantum permutation matrices is as follows:
\begin{enumerate}
\item Pick $n^2$ rank one projections (or rather $n^2$ vectors in $\mathbb C^n$) and put them into an $n\times n$ matrix.
\item Orthonormalize the rows.
\item Orthonormalize the columns.
\item Repeat these two orthonormalizations in an alternating way until the error to the magic unitary condition (Def. \ref{DefMagicUnitary}) is small, or a pre-defined maximal number of iterations has been reached.
\end{enumerate}

The above procedure is reminiscent of the classical Sinkhorn algorithm \cite{sinkhorn1964relationship} for obtaining bistochastic matrices, by normalizing the row sums and the column sums of a matrix with positive entries. We are performing the same alternating minimization procedure, the only difference being that we are considering operator entries instead of scalar elements. The normalization procedure is performed in a way which preserves hermitian (actually positive semidefinite) operators: 
\begin{equation}\label{eq:normalization-S}
X_{ij}' := S^{-1/2} X_{ij} S^{-1/2},
\end{equation}
where $S$ is the matrix corresponding to a row (resp.~ column) sum. Note that in our case the matrix $S$ will be generically positive definite (hence invertible); more general situations would require taking a (Moore-Penrose) pseudo-inverse in the equation above. We initialize our matrix blocks $X_{ij}$ with random Gaussian rank-one projections $X_{ij} = x_{ij} x_{ij}^*$, the vectors $x_{ij} \in \mathbb C^n$ being independent, identically distributed Gaussian vectors. Formally, in the procedure described in detail below, the variable $x$ is a $3$-tensor. The normalization procedure is based on the following easy lemma, which provides an explicit Gram-Schmidt orthonormalization procedure. We put $[n] := \{1, 2, \ldots, n\}$ and we denote the Schatten-2 (or Hilbert-Schmidt) norm on $M_n(\mathbb C)$ by $\|{X}\|_2 := \sqrt{\operatorname{Tr}(X^*X)}$.

\begin{lem}
Let $x_1, \ldots, x_n$ be arbitrary linearly independent vectors in $\mathbb C^n$. Define, for $i \in [n]$, $y_i = S^{-1/2} x_i$, where $S = \sum_{i=1}^n x_i x_i^*$. Then, the vectors $\{y_1
, \ldots, y_n\}$ form an orthonormal basis of $\mathbb C^n$ and we have $\sum_{i=1}^n y_iy_i^*=I_n$.
\end{lem}
\begin{proof}
To begin with, note that the linear independence condition for the $x_i$'s implies that the matrix $S$ is positive definite, hence it admits a square root inverse. It is then clear that the vectors $y_i$ satisfy $\sum_{i=1}^n y_i y_i^* = I_n$; we show next that such vectors must form an orthonormal basis. First, note that, for all $i \in [n]$, 
$$\|y_i\|^2 = \langle y_i , y_i \rangle = \left\langle y_i , \left( \sum_{j=1}^n y_jy_j^* \right)y_i \right\rangle = \sum_{j=1}^n |\langle y_i, y_j \rangle|^2 \geq \|y_i\|^4,$$
implying that $\|y_i\| \leq 1$. On the other hand, we have 
$$n = \operatorname{Tr} I_n = \sum_{j = 1}^n \operatorname{Tr}(y_j y_j^*) = \sum_{j=1}^n \|y_j\|^2 \leq n.$$
This implies that equality holds in all the above inequalities and thus $\|y_i\| = 1$ for all $i$ and $\langle y_i , y_j \rangle = 0$ for all $i \neq j$.
\end{proof}

Note that the iterative scaling procedure we propose will not stop (in the generic case) after a finite number of steps (the same holds true for the original Sinkhorn algorithm). Hence, one needs to stop after reaching a target precision, or after a pre-defined number of iterations (in which case we declare that the algorithm has failed in finding an approximate quantum permutation matrix). We measure the distance between a matrix $X$ and the set of quantum permutation matrices by the error
$${error} := \max\left\{\max_{i=1}^n \left\|\sum_{k=1}^n X_{ik} - I_n\right\|_2  ,  \max_{j=1}^n \left\|\sum_{k=1}^n X_{kj} - I_n\right\|_2  \right \}.$$
When computing the error in the Algorithm \ref{alg:Sinkhorn-Sn-plus} below, note that we only need to take into account the row sums, since this computation follows column normalization, hence the column sums are (close to being) normalized. 

We provide in Algorithm \ref{alg:Sinkhorn-Sn-plus} a detailed description of the algorithm producing quantum permutation matrices. The main functions in the algorithm are the row and column normalizations procedures. We present next the row normalization routine, leaving the case of columns to the reader.

\begin{algorithmic}[1]
\Procedure{NormalizeRows}{$x$}
\For{$i \in [n]$}\Comment{normalize rows}
\State $R_i \gets 0_n$
\For{$j \in [n]$} \Comment{compute row sums}
\State $R_i \gets R_i + x_{ij}x_{ij}^*$
\EndFor
\For{$j \in [n]$} \Comment{update elements}
\State $x_{ij} \gets R_i^{-1/2}x_{ij}$
\EndFor								
\EndFor
\State \Return $x$
\EndProcedure
\end{algorithmic}

\begin{algorithm}
	\caption{Sinkhorn algorithm for $S_n^+$}	
	\begin{algorithmic}[1]
		
		\Procedure{GenerateMagicUnitary}{$n,\epsilon,N_{max}$}
		\State $x \gets \text{random $n \times n \times n$ complex Gaussian tensor}$\Comment{random initialization}
		\State $iter \gets 0$\Comment{iteration counter}
		\Repeat
		\State $iter \gets iter + 1$
		\State $x \gets \textsc{NormalizeRows}(x)$ \Comment{normalize rows}
		\State $x \gets \textsc{NormalizeCols}(x)$ \Comment{normalize columns}
		\State $error \gets 0$\Comment{compute error}
		\For{$i \in [n]$}\Comment{need to check rows only}	
		\State $R_i \gets 0_n$ 
		\For{$j \in [n]$} \Comment{compute row sums}
		\State $R_i \gets R_i + x_{ij}x_{ij}^*$
		\EndFor	
		\State $error \gets \max(error, \|R_i - I_n\|_2)$ \Comment{compute maximal error}
		\EndFor
		\Until{$error < \epsilon$ \textbf{or} $iter > N_{max}$}\Comment{end conditions}
		\If {$error < \epsilon$}
		\State \Return{ $x$} \Comment{quantum permutation found}
		\Else
		\State \Return{$\text{failure}$} \Comment{too many iterations}
		\EndIf
		\EndProcedure
		
	\end{algorithmic}
	\label{alg:Sinkhorn-Sn-plus}
\end{algorithm}

We finish this section by re-stating a conjecture from \cite{banica2017flat}; note that a proof of convergence in the setting where unit rank matrices are replaced by positive definite elements is provided in \cite[Proposition A.1]{benoist2017bipartite}. 

\begin{conj}
For any precision parameter $\epsilon > 0$ and for almost all initializations of the alternating normalization procedure from Algorithm \ref{alg:Sinkhorn-Sn-plus}, the program will terminate sucessfully after a finite number of steps. 
\end{conj}

\subsection{Restriction to uniformly vertex transitive graphs}

Before presenting the algorithm for quantum automorphisms of graphs, Algorithm \ref{alg:Sinkhorn-QAut}, we make in this subsection and in the following one some important observations. First, it turns out that, given the type of quantum permutation matrices we consider, we will have to restrict our attention to graphs of a certain type.

\begin{defn}\label{def:vertex-transitive}
Let  $\Gamma=(V,E)$ be a graph with $V=\{1,\ldots,n\}$.
\begin{itemize}
\item[(a)] The graph is \emph{vertex transitive}, if for all $i=1,\ldots,n$ we find $\sigma_1^i,\ldots,\sigma_n^i\in\Aut(\Gamma)$ such that all entries of the $i$-th row of $\sum_k \sigma_k^i$ are 1. This expresses the property, that each vertex may be mapped to each other vertex by some automorphism.
\item[(b)] The graph is \emph{uniformly vertex transitive}, if there are $\sigma_1,\ldots,\sigma_n\in\Aut(\Gamma)$ such that all entries of  $\sum_k \sigma_k$ are 1. We call the set $\{\sigma_1,\ldots,\sigma_n\}$ a maximal Schur set.
\end{itemize}
\end{defn}

Note that a uniformly vertex transitive graph is in particular vertex transitive, putting all $\sigma_k^i=\sigma_k$. Moreover, the matrix with all entries equal to one is the unit for the Schur product (also called Hadamard product), hence  the name maximal Schur set. Any Cayley graph of a group $G$ (seen as a directed graph with unlabeled edges) is uniformly vertex transitive, its maximal Schur set simply being $G$ itself. More generally, any vertex transitive graph whose automorphism group has the same cardinality as its vertex set, is uniformly vertex transitive.

In \cite{UVT} we introduce the notion uniformly vertex transitive and study it in detail. In particular, we prove the following chain of implications:
\[\textnormal{Cayley graph} \qquad \Longrightarrow \qquad \textnormal{uniformly vertex transitive}\qquad \Longrightarrow \qquad \textnormal{vertex transitive}\]

Both implications are strict, an example for a uniformly vertex transitive graph which is no Cayley graph is the Petersen graph, while an example for a vertex transitive graph which is not uniformly vertex transitive is the line graph of the Petersen graph. See \cite{UVT} for details.

\subsection{A technical lemma}

In order to state our algorithm for finding quantum symmetries of graphs, we need the following technical lemma on the existence of an optimal transform matrix playing the role of the matrix $S$ from \eqref{eq:normalization-S} in the setting of subspaces of $\mathbb C^n$. We denote the group of $r\times r$ complex valued unitary matrices by $U(r)$. By $\mathrm{Diag}_r^{>0}$, we denote the $r\times r$ diagonal matrices with strictly positive diagonal entries. 

For two given self-adjoint matrices $A,B \in M_n(\mathbb C)$, the equation $FAF^* = B$ (in the unknown variable $F \in M_n(\mathbb C)$) has a solution iff $\operatorname{rk} A \geq \operatorname{rk} B$. We shall be interested in finding such a solution of ``minimal disturbance'', i.e.~close to the identity on the support of $A$. For our applications to finding quantum symmetries of graphs, we shall consider the special case where $\operatorname{rk} A = \operatorname{rk} B =: r$ and $B = B^* = B^2$ is an orthogonal projection; note however that the ideas below can be easily modified to cover the general case. If we write $P_A$, resp.~$P_B$, for the orthogonal projections on the supports of $A$, resp.~$B$, it is clear that we can restrict the space of solutions $F$ to matrices satisfying 
\begin{equation}\label{eq:F-restrictions-projection}
P_B F = F P_A = F.
\end{equation}

\begin{lem}\label{lem:soft-normalization}
Let $A=A^*, B=B^*=B^2 \in M_n(\C)$ be matrices of rank $r$ admitting singular value decompositions 
$$A = V_A \Delta_A V_A^* \quad \text{ and } \quad B= V_B I_r V_B^*,$$
for some isometries $V_A, V_B: \C^r \to \C^n$ and some positive diagonal matrix  $\Delta_A \in \mathrm{Diag}^{>0}_r$. Then, the set of solutions $F \in M_n(\C)$ of the equation $FAF^* = B$ satisfying Eq.~\eqref{eq:F-restrictions-projection} is parameterized by unitary matrices $W \in U(r)$ as follows:
$$F_W= V_B W \Sigma V_A^*,$$
with $\Sigma= \Delta_A^{-\frac{1}{2}} \in \mathrm{Diag}^{>0}_r$. Furthermore, for the unitary $W_0$ defined by the polar decomposition
$$V_B^*V_A\Sigma = W_0 |V_B^*V_A\Sigma|,$$ 
it holds that
\begin{align*}
\|F_{W_0} - P_A\|_2 = \min_{W \in U(r)}\|F_W - P_A\|_2.
\end{align*}
\end{lem}

\begin{proof}
First, let us check that the matrices $F_W$ satisfy our matricial equation:
\begin{align*}
F_WAF_W^* &= V_B W\Sigma V_A^*A V_A\Sigma W^*V_B^*\\
&=V_B W\Sigma V_A^*V_A \Delta_A V_A^* V_A\Sigma W^*V_B^*\\
&=V_B W\Sigma \Delta_A \Sigma W^*V_B^*\\
&=V_B WW^*V_B^*\\
&=V_B I_r V_B^*\\
&=B,
\end{align*}
where we have used the fact that $W$ is unitary and $V_{A}$, $V_B$ are isometries. Moreover, since  $P_{A}=V_{A} V_{A}^*$, we also have
$$F_W P_A=  V_B W\Sigma V_A^* V_A V_A^* = F_W,$$
showing that the matrices $F_W$ also satisfy Eq.~\eqref{eq:F-restrictions-projection}. For the reverse direction, assuming $FAF^*=B$, write
$$V_BI_rV_B^*=B=FAF^* = (FV_A\Delta_A^{1/2})(FV_A\Delta_A^{1/2})^* ,$$
so there must exist a unitary operator $W \in U(r)$ such that
$$FV_A\Delta_A^{1/2} = V_B W,$$
proving the claim that $F = F_W$. 

For the statement about the minimizer $W_0$, we have
\begin{align*}
\|F_W - P_A\|^2_2 &= \| V_B W \Sigma V_A^*- V_A V_A^*\|^2_2\\
&=\operatorname{Tr}((V_B W \Sigma V_A^*- V_A V_A^*)(V_A  \Sigma W^* V_B^*- V_A V_A^*))\\
&= \operatorname{Tr}(V_B W \Sigma^2 W^* V_B^*- V_B W\Sigma V_A^*- V_A \Sigma W^* V_B^* + V_A V_A^*).
\end{align*}
Since $\operatorname{Tr}(V_B W \Sigma^2 W^* V_B^*) = \operatorname{Tr}(\Sigma^2)$ and $\operatorname{Tr}(V_A V_A^*) = r$ do not depend on $W$, we only need to minimize $\operatorname{Tr}(- V_BW\Sigma V_A^*- V_A \Sigma W^* V_B^*) = - 2 \operatorname{Re} \operatorname{Tr}(V_B W \Sigma V_A^*)$. We have
\begin{align*}
\min_{W \in U(r)} -2\operatorname{Re} \operatorname{Tr}(V_B W\Sigma V_A^*)&=-2\max_{W \in U(r)} \operatorname{Re} \operatorname{Tr}(V_B W\Sigma V_A^*)\\
&=-2 \max_{W \in U(r)} \operatorname{Re} \operatorname{Tr}(W \Sigma V_A^*V_B).
\end{align*}
Now, decomposing $\Sigma V_A^*V_B =: X = U_X \Sigma_X V_X^*$ with $U_X, V_X\in U(r)$ we deduce
\begin{align*}
\max_{W \in U(r)} \operatorname{Re} \operatorname{Tr}(W\Sigma V_A^* V_B)&= \max_{W \in U(r)} \operatorname{Re} \operatorname{Tr}(W
U_X \Sigma_X V_X^*)\\
&= \max_{W \in U(r)} \operatorname{Re} \operatorname{Tr}(V_X^* W U_X \Sigma_X)\\
&= \max_{\tilde{W} \in U(r)} \operatorname{Re} \operatorname{Tr}(\tilde{W}\Sigma_X),
\end{align*}
since $V_X^* W U_X$ is unitary. Since $\Sigma_X$ is non-negative, by the matrix H\"older inequality \cite[Chapter IV]{bhatia1997matrix}, we obtain the maximum by choosing $\tilde{W} = I_r$:
$$\operatorname{Re} \operatorname{Tr}(\tilde{W}\Sigma_X) \leq | \operatorname{Tr}(\tilde{W}\Sigma_X)| \leq \|W\|_\infty \|\Sigma_X\|_1 = \operatorname{Tr}\Sigma_X.$$
This yields that the minimum is attained at $F_W$ with $W_0 = V_X U_X^*$.
\end{proof}

\subsection{The algorithm for finding quantum symmetries of graphs: Algorithm \ref{alg:Sinkhorn-QAut}}

In order to check whether a given graph has quantum symmetries, we need to find noncommutative quantum permutation matrices respecting $uA=Au$, in the sense of Lemma \ref{LemQPerm}. 
We thus need to include in Algorithm \ref{alg:Sinkhorn-Sn-plus} the condition $u A=A u$, which is a non-trivial task.

We shall use the same idea of Sinkhorn alternating scaling, with an important twist that we explain next. In the current setting, we are looking for block matrices $X \in M_n(M_n(\mathbb C))$ having unit-rank blocks $X_{ij} = x_{ij} x_{ij}^*$ which should satisfy the following three conditions: 
\begin{enumerate}
\item Row normalization:
$$\forall i \in [n], \qquad \sum_{j = 1}^n X_{ij} = I_n$$
\item Column normalization:
$$\forall j \in [n], \qquad \sum_{i = 1}^n X_{ij} = I_n$$
\item The matrix $X$ encodes a graph symmetry:
$$ X(A \otimes I_n) = (A \otimes I_n)X,$$
where $A$ is the adjacency matrix of the graph $\Gamma = (V,E)$; we shall sometimes abuse notation and write simply $XA = AX$ for the condition above. 
\end{enumerate}
A first idea to produce matrices $X$ satisfying the above three properties is to use Sinkhorn alternating normalizations, enforcing at each step one of the three conditions above. It turns out that inserting the symmetry condition above directly in Algorithm \ref{alg:Sinkhorn-Sn-plus} does not yield a converging algorithm. One can understand this negative result as a consequence of the fact that enforcing the third condition above destroys the magic unitary property of the matrix $X$, rendering the algorithm unstable. 

To circumvent this problem, we introduce a parameter $\tau$ in the algorithm, which will measure the intensity with which the symmetry condition $XA = AX$ will be enforced. The parameter $\tau \in [0,1]$ will be chosen empirically, depending on the graph $\Gamma$ at hand. A small value of $\tau$ will correspond to a ``soft'' application of the normalization $XA = AX$, while the maximal value $\tau = 1$ corresponds to the strict enforcing of the symmetry condition. In practice, we use the value of $\tau$ to interpolate between the initial and the target conditions:
$$\mathrm{Target}^{(\tau)} = \tau \cdot \mathrm{Target} + (1-\tau) \cdot \mathrm{Initial}.$$ 

Let us now describe in detail the normalization procedures used in Algorithm \ref{alg:Sinkhorn-QAut}. The main idea is to integrate the symmetry condition $XA = AX$ (applied with the softening parameter $\tau$) into the usual row and column normalization routines from Algorithm \ref{alg:Sinkhorn-Sn-plus}. To do this, note that the $(i,j)$ element of the equality $XA = AX$ reads
$$\sum_{k\, : \, (k,j) \in E} X_{ik} = \sum_{l \, : \, (i,l) \in E} X_{lj}.$$
Let us assume that we want to perform row normalization on the matrix $X$, and we would like to normalize the sum of the $i$-th row of $X$ to the identity. It turns out that we could impose a stronger constraint than this, by asking, for a fixed $j$, that the following \emph{two} conditions hold:
\begin{align*}
\sum_{k\, : \, (k,j) \in E} X'_{ik} &= \sum_{l \, : \, (i,l) \in E} X_{lj} \\
\sum_{k\, : \, (k,j) \notin E} X'_{ik} &= \sum_{l \, : \, (i,l) \notin E} X_{lj}.
\end{align*}
The two constraints above imply together that the $i$-th row of the updated matrix $X'$ is normalized to the identity, given that the $j$-th column of the original matrix $X$ was normalized (which is true during the execution of our algorithm, since row and column normalizations are alternating). 
\begin{algorithm}
	\caption{Sinkhorn algorithm for $\mathrm{QAut}(\Gamma)$}
	\begin{algorithmic}[1]
		\Procedure{GenerateQAut}{$\Gamma,\tau,\epsilon, N_{max}$}
		\State $x \gets \text{random $n \times n \times n$ complex Gaussian tensor}$\Comment{random initialization}
		\State $iter \gets 0$
		\Repeat
		\State $iter \gets iter + 1$
		\State $x \gets \textsc{NormalizeRowsSoft}(x,\Gamma, \tau)$ \Comment{normalize rows}
		\State $x \gets \textsc{NormalizeColsSoft}(x, \Gamma, \tau)$ \Comment{normalize columns}
		
			\State $errorMagic \gets 0$\Comment{compute error w.r.t.~magic property}
\For{$i \in [n]$}
\State $R_i \gets \sum_j x_{ij}x_{ij}^*$ \Comment{row sum after col.~normalization}
\State{$errorMagic \gets \max(errorMagic,  \|R_i - I_n\|_2)$}
\EndFor
\State $errorComm \gets 0$\Comment{compute error w.r.t.~symmetry}
\For{$(i,j) \in [n]^2$}\Comment{check each element of $XA - AX$}	
\State{$errorMagic \gets \max(errorMagic, \|(XA-AX)_{ij}\|_2)$}
\EndFor
\State $error \gets \max(errorMagic, errorComm)$ \Comment{total error}	
\Until{$error < \epsilon$ \textbf{or} $iter > N_{max}$}\Comment{end conditions}
\If {$error < \epsilon$}
	\State \Return{ $x$} \Comment{quantum graph symmetry found}
\Else
	\State \Return{$\text{failure}$} \Comment{too many iterations}
\EndIf
\EndProcedure
	\end{algorithmic}
\label{alg:Sinkhorn-QAut}
\end{algorithm}

The row and column normalizations are done in a soft way, as described in Lemma \ref{lem:soft-normalization}. The pseudo-code is given below for the row case; we leave the case of column normalization to the reader, as it is very similar. 

	\begin{algorithmic}[1]
\Procedure{NormalizeRowsSoft}{$x,\Gamma,\tau$}
\For{$i \in [n]$}
\State $j \gets \textbf{random}([n])$ \Comment{pick a random column}
\State $R_i \gets 0_n$ 
\State $R_i^\perp \gets 0_n$
\For{$k \in [n]$} \Comment{compute row sums}
\If{$(k,j) \in E(\Gamma)$}
\State $R_i \gets R_i + x_{ik}x_{ik}^*$
\Else
\State $R_i^\perp \gets R_i^\perp + x_{ik}x_{ik}^*$
\EndIf
\EndFor
\State $C_j \gets 0_n$ 
\State $C_j^\perp \gets 0_n$
\For{$k \in [n]$} \Comment{compute column sums}
\If{$(i,k) \in E(\Gamma)$}
\State $C_j \gets C_j + x_{kj}x_{kj}^*$
\Else
\State $C_j^\perp \gets C_j^\perp + x_{kj}x_{kj}^*$
\EndIf
\EndFor		
\State{$(\hat R_i, \hat R_i^\perp) \gets \textsc{NormalizeSum}(R_i,R_i^\perp)$}\Comment{$\hat R_i+ \hat R_i^\perp = I_n$}
\State{$C_j^{(\tau)} \gets \textsc{InterpolateProjections}(\hat R_i, C_j, \tau)$} 
\State{$F \gets \textsc{OptimalTansformation}(R_i,C_j^{(\tau)})$}
\State{$C_j^{\perp(\tau)} \gets \textsc{InterpolateProjections}(\hat R_i^\perp, C_j^\perp, \tau)$} 
\State{$F^\perp \gets \textsc{OptimalTansformation}(R_i^\perp,C_j^{\perp(\tau)})$}
\For{$k \in [n]$} \Comment{update elements}
\If{$(k,j) \in E(\Gamma)$}
\State $x_{ik} \gets Fx_{ik}$
\Else
\State $x_{ik} \gets F^\perp x_{ik}$
\EndIf
\EndFor								
\EndFor	
\State \Return{x}
\EndProcedure		
		
\end{algorithmic}

The procedure above makes use of three subroutines, \textsc{NormalizeSum}, \textsc{InterpolateProjections}, and \textsc{OptimalTransformation}, that we describe next. The first subroutine performs a standard normalization:
$$\textsc{NormalizeSum}(X, Y) = (S^{-1/2} X S^{-1/2}, S^{-1/2} Y S^{-1/2}), \quad \text{ with } S:= X+Y,$$
where the inverse square roots are taken in the Moore-Penrose sense (although the matrix $S$ will be generically invertible in our scenario). Note that the calls to the \textsc{InterpolateProjections} routine on lines 19 and 21 is made with projections $\hat R_i$, $C_j$ (resp.~$\hat R_i^\perp$, $C_j^\perp$) of equal rank. This follows from the fact that any vertex transitive graph is regular and from the running assumption that the vectors $x_{ij}$ appearing in each row or column are in general position. 

The second subroutine, \textsc{InterpolateProjections}, takes as an input two projections $P$, $Q$ of equal rank $r$, and an interpolation parameter $\tau \in [0,1]$. It outputs the closest rank-$r$ projection to $A := \tau P + (1-\tau)Q$. The way this approximation works is via the eigenvalue decomposition of $A$: the closest rank-$r$ projection to $A$ is obtained by setting the top $r$ eigenvalues of $A$ to 1, and all the others to 0. 

The third subroutine, \textsc{OptimalTransformation}, computes the optimal matrix $F$ from Lemma \ref{lem:soft-normalization}, with the choice of the unitary $W$ described there. In practice, it uses the singular value decomposition several times, the implementation being straightforward. 

\section{Experiments and discussion of the results}
\label{sec:experiments}

We gather in this section the numerical experiments we realized using Algorithms \ref{alg:Sinkhorn-Sn-plus}
 and \ref{alg:Sinkhorn-QAut}. We comment on the theoretical implications of these experiments, and how the results can guide future analytical research. 
 All the experimental data and runtimes in this paper was produced in \texttt{Octave 4.0.3.}~(Debian GNU/Linux 9 64-bit), running on an  Intel Core i5-7600 CPU @ 3.50GHz with 16 GB of memory. The code, available in the supplementary material of the arXiv submission, is compatible with \texttt{MATLAB}. 

\subsection{Generation of quantum permutation matrices: Algorithm 1}

We analyze in this section the performance of Algorithm \ref{alg:Sinkhorn-Sn-plus} for generating quantum permutation matrices (recall: no adjacency matrix $A$ and no graphs are involved here). Let us start by pointing out that since the initialization of the matrix of unit-rank projections is random (we use i.i.d.~random Gaussian vectors), the running time (i.e.~the number of iterations), as well as the probability of success (whether the algorithm reaches the desired precision after a given number of iterations) are random variables. We plot in Figure \ref{fig:alg-1-iterations} the data for $n=4$ and $n=10$.

\begin{figure}[H]
    \centering
    \includegraphics[scale=0.55]{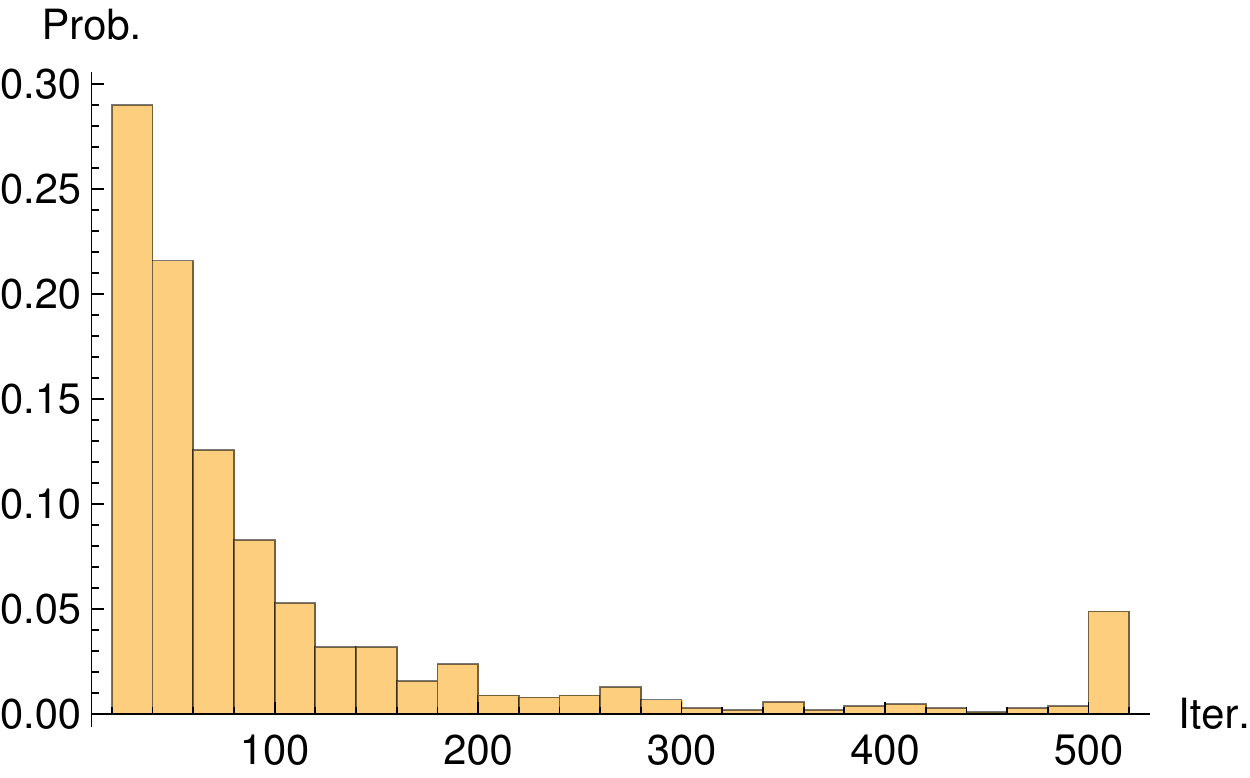} \qquad
    \includegraphics[scale=0.55]{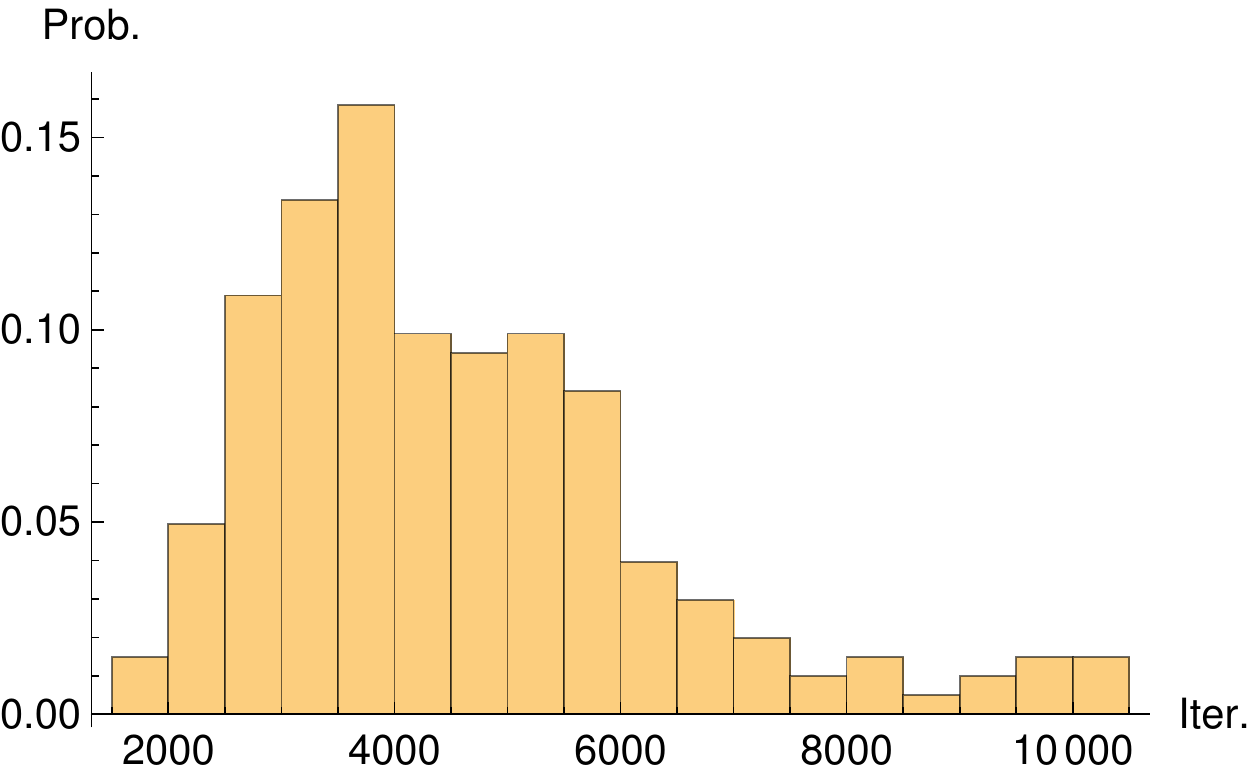}
    \caption{Histograms for the number of iterations Algorithm \ref{alg:Sinkhorn-Sn-plus} needs to produce a quantum permutation matrix. On the left, $n=4$, the precision is set to $10^{-6}$, and the maximal number of iterations is $N_{max} = 500$; on the right, $n=10$, the precision is $10^{-2}$, and $N_{max} = 10^6$. The bars at the right edge of the histograms correspond to failed executions: the desired precision had not been reached after $N_{max}$ iterations. The probability of failure reads $\mathbb P_{n=4}(\text{failure}) = 4.9\%$ and $\mathbb P_{n=10}(\text{failure}) = 0\%$ respectively. Experiments have been ran $10^3$ times for both values of $n$.} 
    \label{fig:alg-1-iterations}
\end{figure}

Algorithm \ref{alg:Sinkhorn-Sn-plus} was used to produce the rows corresponding to the complete graphs $K_{4}$, $K_5$, $K_6$ in Table \ref{tab:graphs-q-sym} (note, for complete graphs, the relation $uA=Au$ is no additional condition). We present in Table \ref{tab:failure-Kn} the probability that the algorithm fails to find an $\epsilon$-good approximation before $N_{max}$ iterations, for $n=4,5,6$, showing the influence of the maximal number of steps the algorithm is allowed to take to reach the desired accuracy on the success probability. 

\begin{table}
\begin{tabular}{|c|c|c|c|}
\hline
$n$ & $\epsilon$ & $N_{max}$ & $\mathbb P(\text{failure})$ \\ \hline \hline
4 & $10^{-6}$ & 500 & 4.9\% \\ \hline
5 & $10^{-6}$ & 2000 & 0.2\% \\ \hline
6 & $10^{-6}$ & 5000 & 11.3\% \\ \hline
\end{tabular}
\caption{Probability of failure of Algorithm \ref{alg:Sinkhorn-Sn-plus} to produce an $\epsilon$-approximation of a magic unitary in $N_{max}$ iterations, as a function of $n$, $\epsilon$, and $N_{max}$.}
\label{tab:failure-Kn}
\end{table}

\subsection{Generation of quantum symmetries of graphs: Algorithm 2}

In this section we discuss certain aspects as well as the produced data of  Algorithm \ref{alg:Sinkhorn-QAut}. Most of our experiments are summarized in Table \ref{tab:graphs-q-sym}. 

We first address the parameter $\tau$. In a sense, this parameter tells our algorithm how fast we want to insert the relation $u A = A u$ at the cost of destroying the normalized rows and columns. There is no natural candidate for an optimal $\tau$. To get an idea, we run the algorithm several times for different values of $\tau$ and compare how fast it terminates (see Figure \ref{fig:alg-2-tau}). Here we have the values $\tau = n\cdot 0.1$ for $n \in \{0,\dots, 10\}$ on the $x$-axis and the number of steps the algorithm takes to terminate on the $y$-axis. We choose to discuss the choice of the hyper-parameter $\tau$ on two graphs, the Petersen and the Cube $Q_3$ graph. These graphs, depicted in Figure \ref{fig:Petersen-Cube} are well-known in graph theory, and they are interesting from our perspective because they are uniformly vertex transitive and they lack (Petersen) resp. have (Cube) quantum symmetry. 

\begin{figure}[H]	
    \centering
    \includegraphics[scale=0.55]{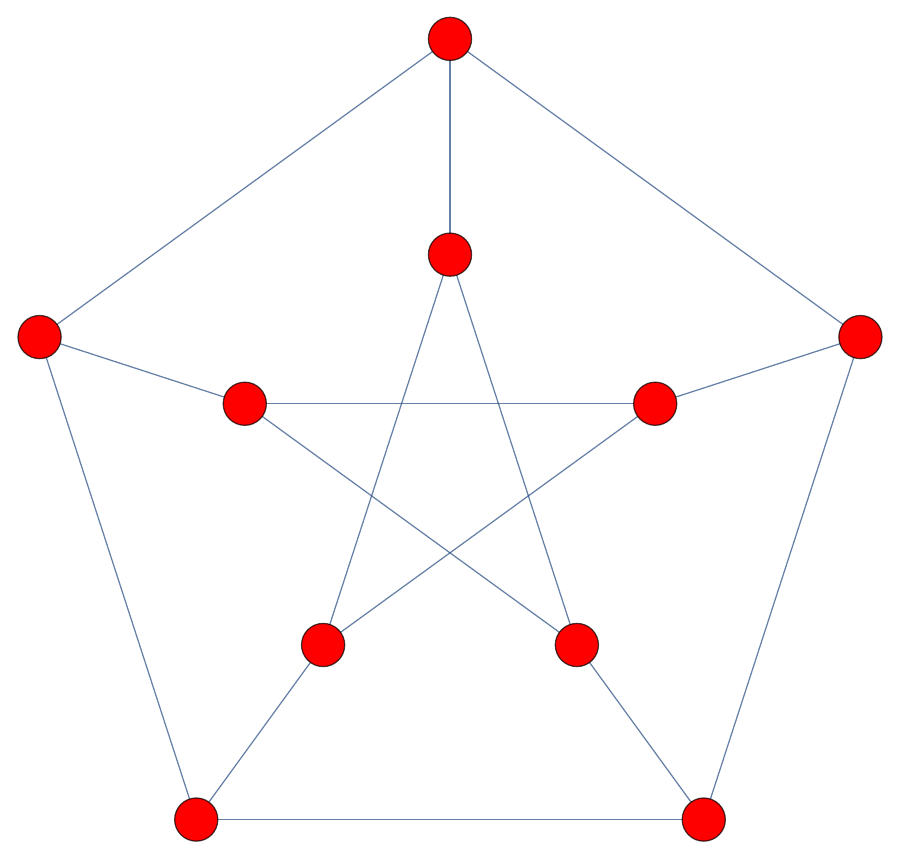} \qquad\qquad\qquad \includegraphics[scale=0.55]{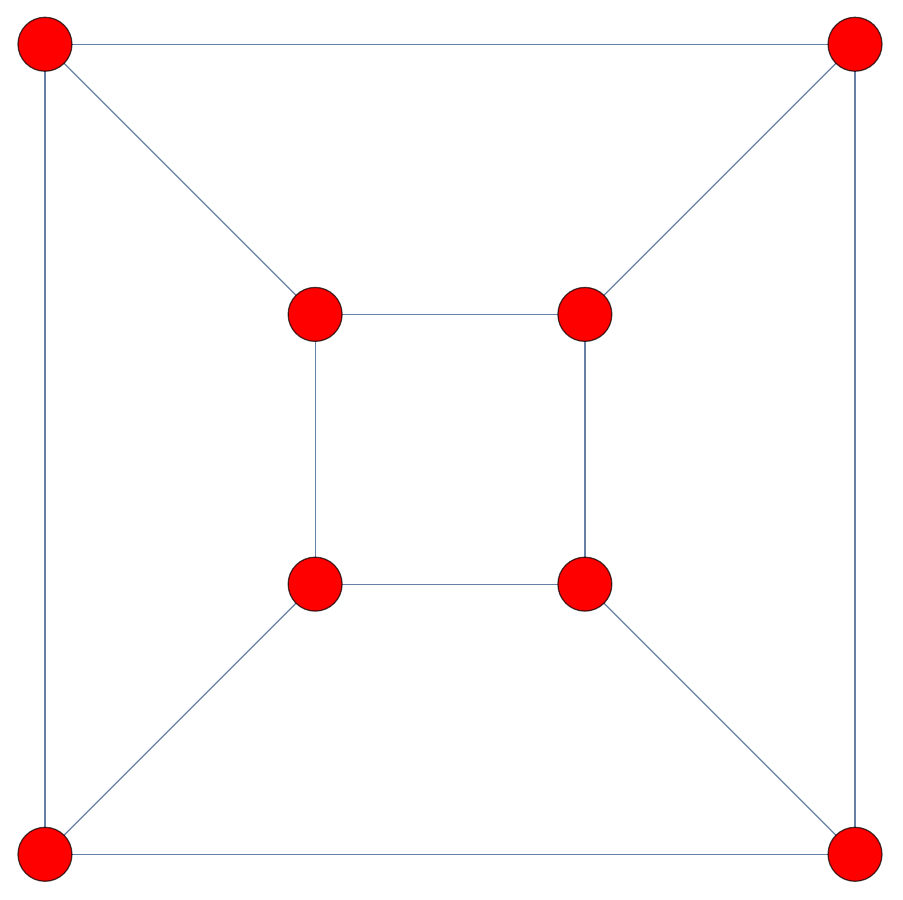}
    \caption{The Petersen (left) and the Cube $Q_3$ graphs.}
    \label{fig:Petersen-Cube}
\end{figure}

\begin{figure}[H]	
    \centering
    \includegraphics[scale=0.55]{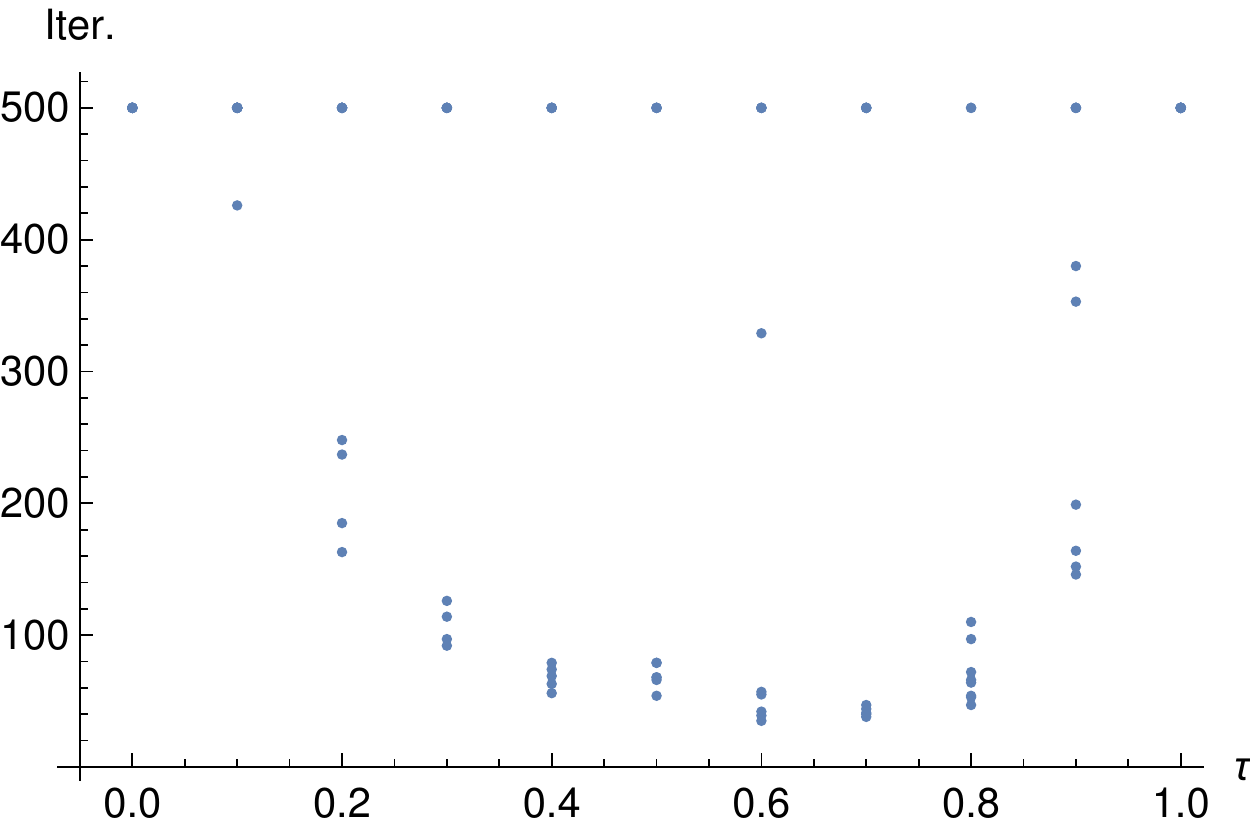} \qquad     \includegraphics[scale=0.55]{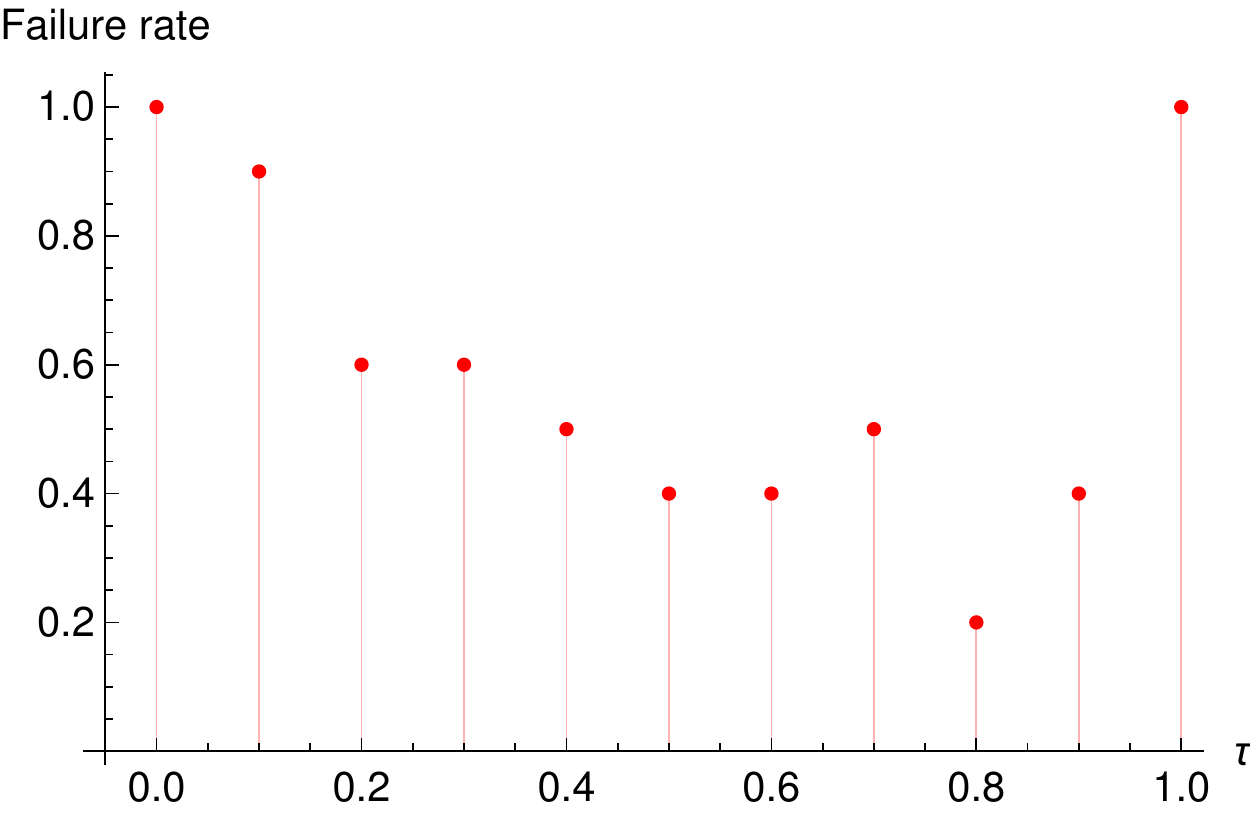} \\
    \includegraphics[scale=0.55]{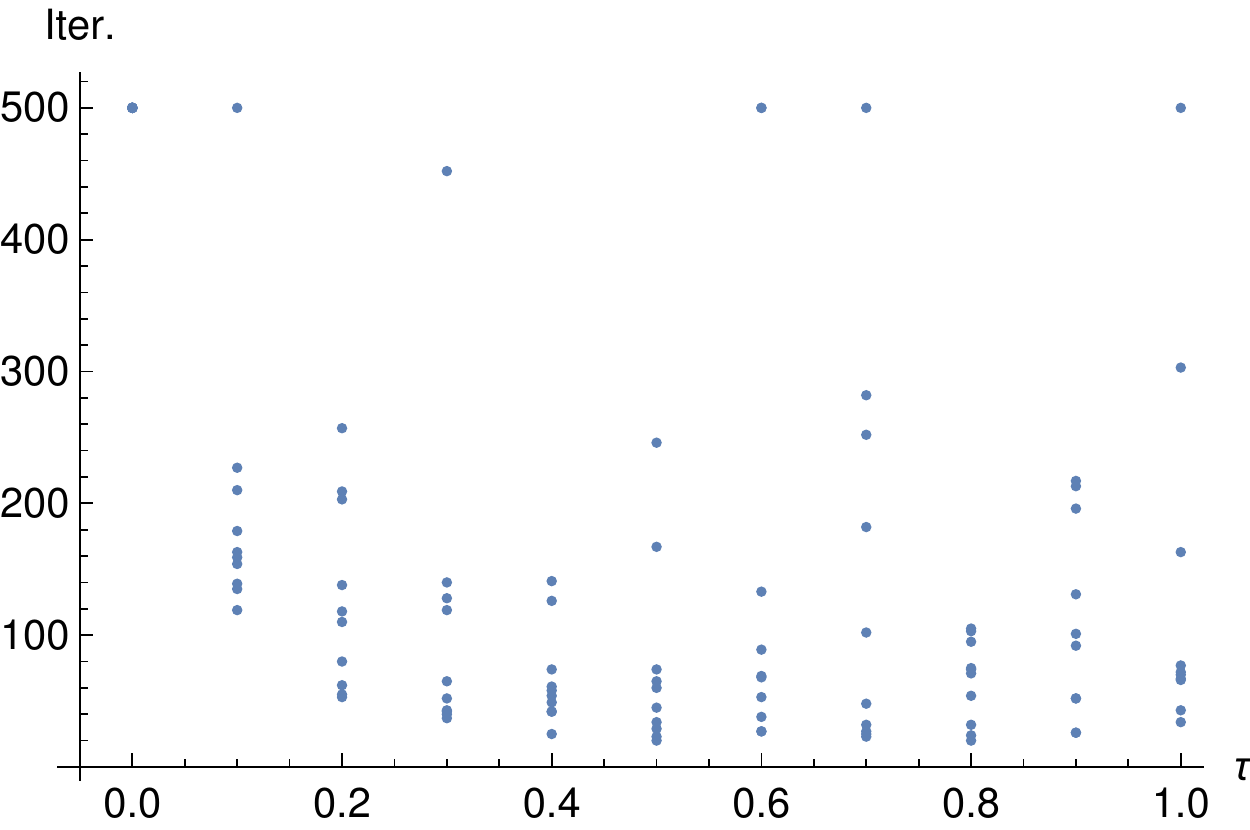} \qquad     \includegraphics[scale=0.55]{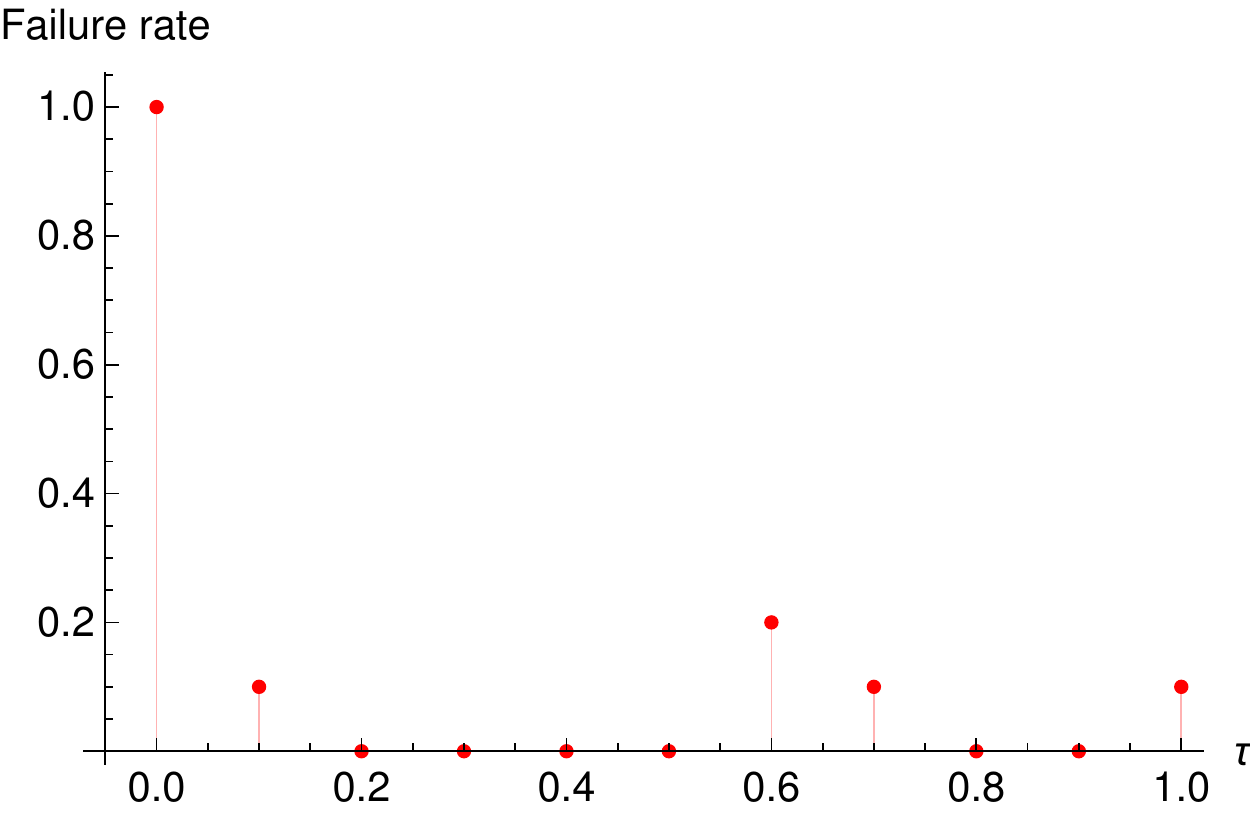} \\
    \caption{Number of iterations the algorithm needs to find an approximative quantum symmetry for the Petersen graph (first row) and the Cube $Q_3$ graph (second row). 
}
    \label{fig:alg-2-tau}
\end{figure}

Looking at the produced data, we choose $\tau =0.5$ for both the Petersen graph the Cube graph. We use this procedure also for other graphs to get some estimates for $\tau$, finding in general that a value of $\tau = 0.5$ works well in practice. 

Now, we run our algorithm for the chosen $\tau$ many times. We plot the probability that the algorithm stops ($y$-axis) at a certain number of steps ($x$-axis), for the Cube graph and the Petersen graph. We do this first for the Cube graph, see Figure \ref{figure2}. 

\begin{figure}[H]
    \centering
    \includegraphics[page=1, scale=0.5]{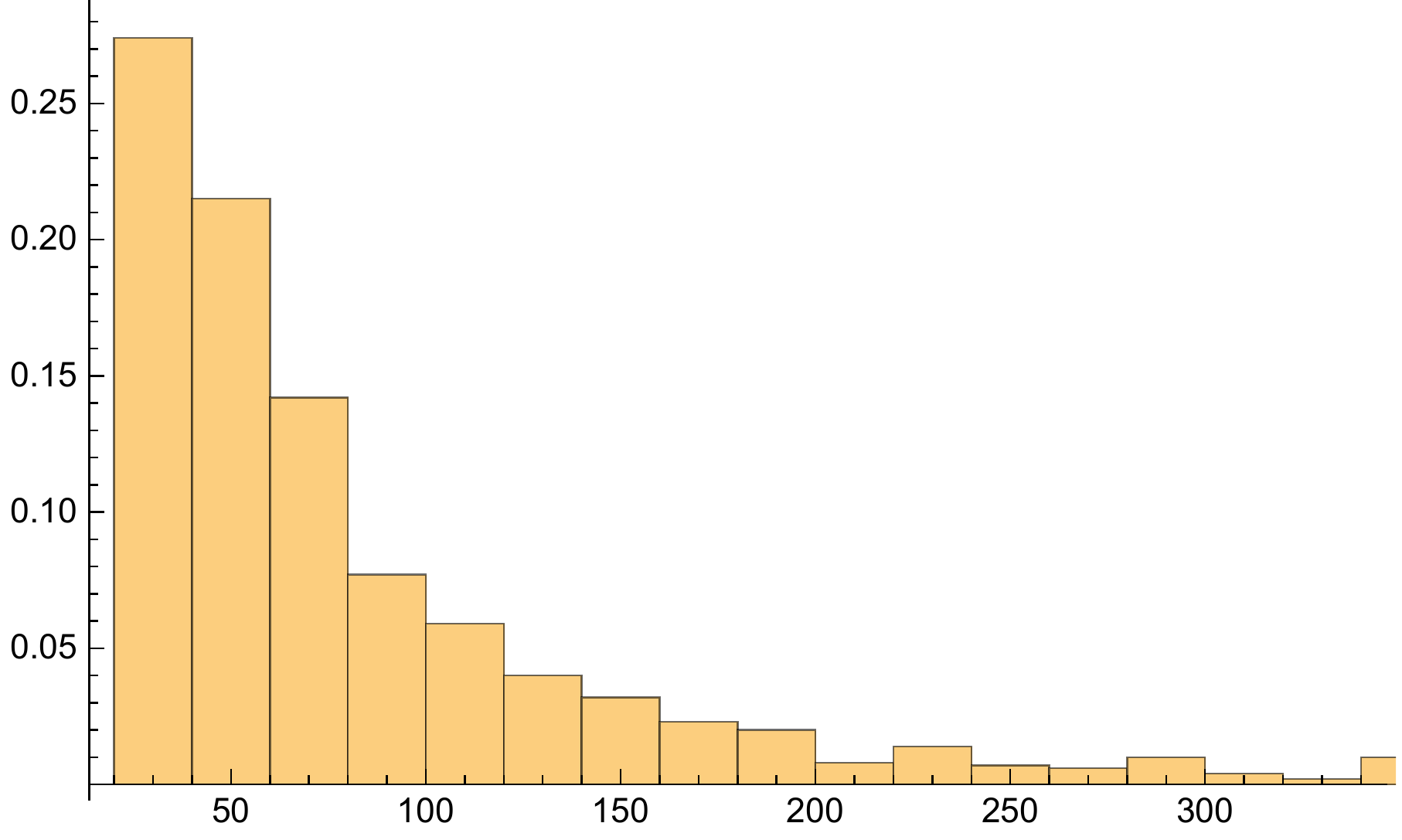}
    \caption{Steps after the algorithm terminates for the Cube $Q_3$ graph.}
    \label{figure2}
\end{figure}

The probability that the algorithm stops goes down quite smoothly. For the Petersen graph, however, the picture looks quite differently. 

\begin{figure}[H]
    \centering
    \subfloat{{\includegraphics[page=2, trim = 0mm 125mm 0mm 0mm, clip, width=0.45\textwidth]{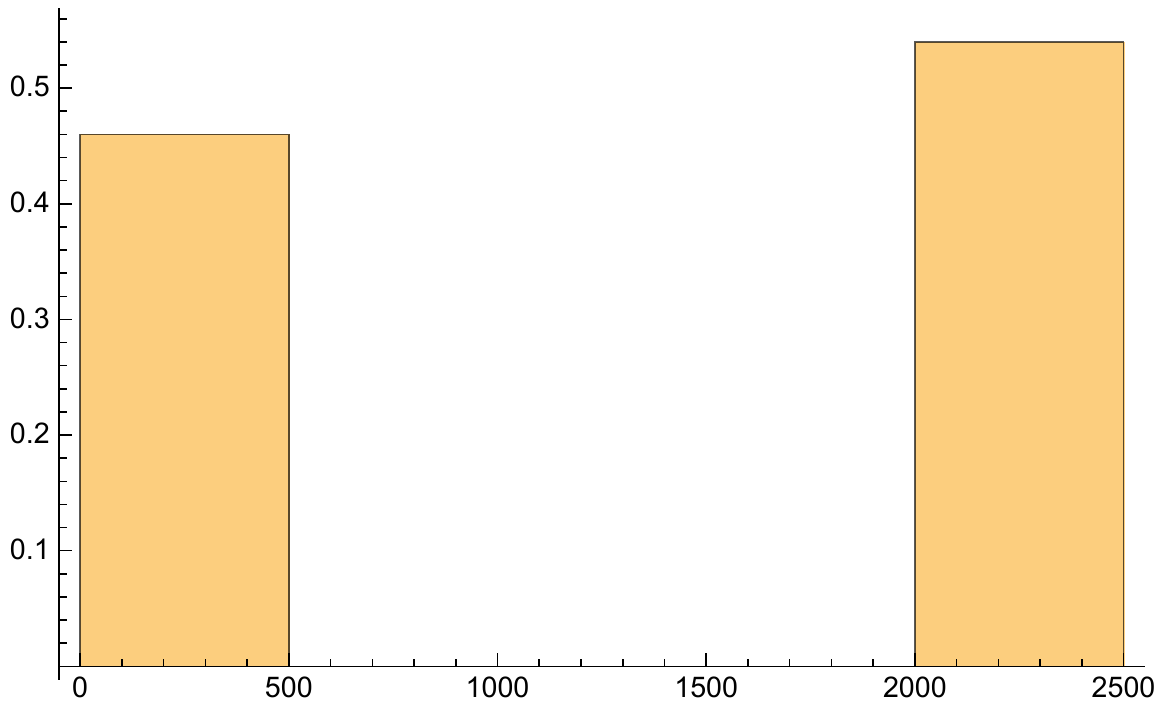} }}
    \qquad
    \subfloat{{\includegraphics[page=2, trim = 0mm 2mm 0mm 120mm, clip, width=0.45\textwidth]{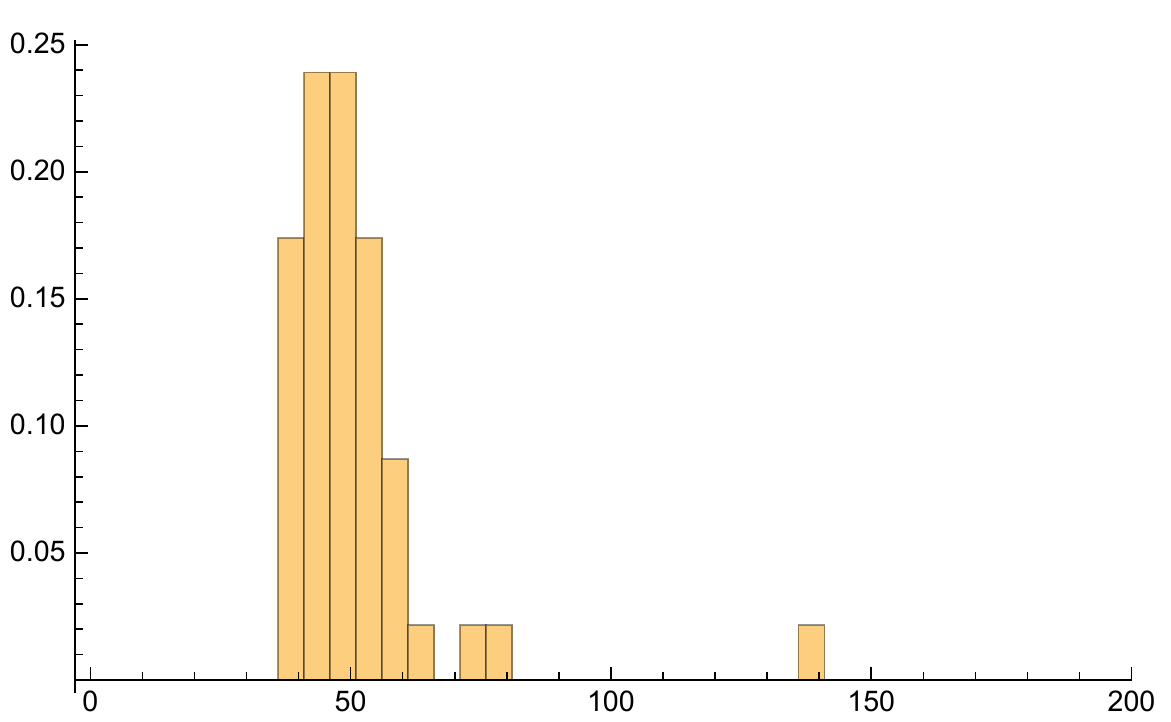} }}
    \caption{Steps after the algorithm terminates for the Petersen graph, on a large scale on the left, zoomed in on the right.  }
    \label{figure3}
\end{figure}

Note that we stopped the algorithm after $2000$ steps. Thus, Figure \ref{figure3} tells us that the algorithm either terminates quickly, after around $50$ steps, or it does not converge before the $2000$ allowed steps. One could try to explain the difference in the Figures \ref{figure2} and \ref{figure3} by saying that the algorithm is more likely to terminate with ``bad'' starting data for the Cube graph, because in contrast to the Petersen graph, the Cube graph has quantum symmetry, see \cite{banica2007hyperoctahedral,schmidt2018petersen}. We do not know if the algorithm always behaves similar to those two cases for graphs with or without quantum symmetry. 

Another important question is the following: How good is  Algorithm \ref{alg:Sinkhorn-QAut}  in detecting quantum symmetry correctly, i.e.~are the matrices we obtain from the algorithm non-commutative exactly for graphs with quantum symmetry and commutative for graphs without quantum symmetry? Looking at Table \ref{tab:graphs-q-sym}, we see that this is always the case for the graphs with at most 16 vertices whenever we may confirm our findings by theoretical means, i.e.~whenever the existence or absence of quantum symmetry has been proven by other means. Thus, first running the algorithm several times for the same graph $\Gamma$, we can guess whether or not this graph has quantum symmetry in order to obtain a hint whether we should aim at proving or rather disproving the existence of quantum symmetry in this case.

For the Hamming graph $H(2,4)$ on 16 vertices, the algorithm returns twice a commutative matrix despite the fact that this graph has quantum symmetry. For all other graphs with quantum symmetry, this never happens. This raises the question how $\Aut(H(2,4))$ sits inside $\QAut(H(2,4))$, in some sense -- possibly, $\Aut(H(2,4))$  is ``relatively large''.

Finally, let us outline a possible way of circumventing the restrictions that our graphs should be uniformly vertex transitive. Instead of working with blocks of rank one, we could consider a more general case, where the blocks of the magic unitary encoding the graph symmetry could have general ranks. These ranks should be pre-assigned and they should satisfy some compatibility condition imposed by the adjacency matrix $A$ of the graph $\Gamma$. Note that for any vertex transitive graph, there is a $k\in\mathbb N$ and $I\subset \Aut(\Gamma)$ such that $\sum_{\sigma\in I}\sigma=kI_n$. For instance, take $I=\Aut(\Gamma)$ and $|\Aut(\Gamma)|=k|V|$. Thus, uniform vertex transitive graphs are just those where $k=1$ is the minimal such number.
We leave such generalizations of Algorithm \ref{alg:Sinkhorn-QAut} for future work. 

\bigskip

\noindent\textit{Acknowledgments.} 
The authors would like to thank Piotr So{\l}tan and Laura Man\v{c}inska for very useful discussions around some of the topics of this work. They would also like to acknowledge the hospitality of the Mathematisches Forschungsinstitut Oberwolfach during the workshop 1819 ``Interactions between Operator Space Theory and Quantum Probability with Applications to Quantum Information'', where some of this research has been done. 
I.N.'s research has been supported by the ANR, projects {StoQ} (grant number ANR-14-CE25-0003-01) and {NEXT} (grant number ANR-10-LABX-0037-NEXT). S.Sch and M.W. have been supported by the \emph{DFG} project \emph{Quantenautomorphismen von Graphen}. M.W. has been supported by \emph{SFB-TRR 195}.

\bibliographystyle{alpha}
\bibliography{qg-qit}

\end{document}